\documentclass[11pt,A4]{article}
\usepackage{amssymb}
\usepackage{amsmath,amsthm}
\usepackage{amsfonts}
\usepackage{graphicx}
\usepackage[latin5]{inputenc}
\usepackage{pstricks}

\newtheorem{theorem}{Theorem}
\newtheorem{lemma}{Lemma}
\newtheorem{example}{Example}

\newtheorem{remark}{Remark}

\begin{document}
\title{Graph-Directed Fractal Interpolation Functions \date{}}

\author{Ali DENİZ \footnote{Department of Mathematics, Anadolu University, Eskisehir, Turkey. Email: adeniz@anadolu.edu.tr } \  and Yunus ÖZDEMİR \footnote{Department of Mathematics, Anadolu University, Eskisehir, Turkey. Email: yunuso@anadolu.edu.tr }}

\maketitle
\begin{abstract}
It is known that there exists a function interpolating a given data set such that the graph of the function is the attractor of an iterated function system which is called fractal interpolation function. We generalize the notion of fractal interpolation function to the graph-directed case and prove that for a finite number of data sets there exist interpolation functions each of which interpolates corresponding data set in $\mathbb{R}^2$ such that the graphs of the interpolation functions are attractors of a graph-directed iterated function system.
\end{abstract}



\section{Introduction}
Barnsley introduced fractal interpolation functions (FIF) using iterated function system (IFS) theory which is an important part of fractals (see \cite{BarnsleyMak}, \cite{BarnsleyBook}, \cite{navassurvey} for further information).  He showed that there exists a function interpolating a given data set such that the graph of this function is the attractor of an IFS. In recent decades it has been widely used in various fields such as approximation theory, image compression, modeling of signals and in many other scientific areas (\cite{chen},\cite{malysz},\cite{navascue1},\cite{navascue2},\cite{xiu}).

Let us summarize the notions  of IFS and FIF firstly. An IFS  is a finite collection of contraction mappings  \mbox{ $\varphi_i: X \to X \, (i=1,...,n)$ } on a complete metric space $X$. It is known that there exists a unique nonempty compact set $A$ satisfying \[A=\bigcup \limits_{i=1}^{n} \varphi_i(A) \subset X \] which is called the attractor of the IFS (see \cite{Hut}).

A data set is a set of points
\begin{equation*}
\mathfrak{D}=\{(x_0,F_0),(x_1,F_1),\dots,(x_N,F_N)\}\subset \mathbb{R}^2
\end{equation*}
such that $x_0<x_1<\cdots<x_N$ with $N\geq 2$. An interpolation function corresponding to the data set $\mathfrak{D}$ is a continuous function \mbox{$f:[x_0,x_N]\to \mathbb{R}$} whose graph passes through the points of $\mathfrak{D}$, i.e. $f(x_i)=F_i$ for all $i=0,1,\dots,N$.

Barnsley constructed an iterated function system whose attractor is an interpolation function corresponding to the data set $\mathfrak{D}$ as follows:

 Let $\omega_i: \mathbb{R}^2 \to \mathbb{R}^2$ be affine functions of the form \[\omega_i(x,y)= \left(
   \begin{array}{cc}
     a_i & 0 \\
     c_i & d_i\\
   \end{array}
 \right)
  \left(
    \begin{array}{c}
      x\\
      y \\
    \end{array}
  \right)+\left(
    \begin{array}{c}
      e_i\\
      f_i \\
    \end{array}
  \right)
  \] where the $a_i,c_i,d_i,e_i,f_i$ are real numbers for $i=1,2,...,N$. Using the constraints \[\omega_i(x_0,F_0)=(x_{i-1},F_{i-1}) \ \mbox{ and }\ \omega_i(x_N,F_N)=(x_i,F_i),\]
  one can obtain the coefficients $a_i,c_i,d_i,e_i,f_i$ easily in terms of data points $\{(x_i,F_i)\}_{i=0}^N$ and the parameters $d_i$ (so-called vertical scaling factor) for $i=1,2,...,N$ as
  \begin{eqnarray*}
  a_i&=&\frac{x_i-x_{i-1}}{x_N-x_0}\\
  e_i&=&\frac{x_N x_{i-1}-x_0x_{i}}{x_N-x_0}\\
  c_i&=&\frac{F_i-F_{i-1}}{x_N-x_0}-d_i \frac{F_N-F_{0}}{x_N-x_0}\\
  f_i&=&\frac{x_N F_{i-1}-x_0F_{i}}{x_N-x_0}-d_i \frac{x_N F_0-x_0F_N}{x_N-x_0}
  \end{eqnarray*}
  (see \cite{BarnsleyMak}, \cite{BarnsleyBook} for details).

  In \cite[Teorem 2.1, p.217]{BarnsleyBook}, it is shown that $\{ \mathbb{R}^2; \omega_1,...,\omega_N\}$ is an  iterated function system (associated with the data set $\mathfrak{D}$) in case of $|d_i|<1$ for $i=1,...,N$.

  The attractor of the IFS $\{ \mathbb{R}^2; \omega_1,...,\omega_N\}$ constructed above is the graph of an interpolation function which is called (affine) fractal interpolation function (FIF) corresponding to the data set $\mathfrak{D}$. Note that there exist infinitely many fractal interpolation functions  depending on the parameters $d_i$ with $|d_i|<1$ for $i=1,...,N$.

  \begin{example}\label{exampleifs}
  Consider the IFS (constructed as above) associated with the data set
  \[\{(0,0), (3,5), (6,4), (10,1)\}\]
  where we choose $d_1=0.25, d_2=0.5, d_3=0.25$. The IFS consists of the contractions
  \begin{eqnarray*}
  \omega_1(x,y)&=&(0.3 x, 0.475x+0.25y)\\
  \omega_2(x,y)&=&(0.3 x +3, -0.15x+0.5y+5)\\
  \omega_3(x,y)&=&(0.4 x+6, -0.325x+0.25y+4)
  \end{eqnarray*}
  and its attractor is a FIF whose graph depicted in Figure \ref{fifexample}.

  \begin{figure}[h]
  \begin{center}\includegraphics[scale=0.3]{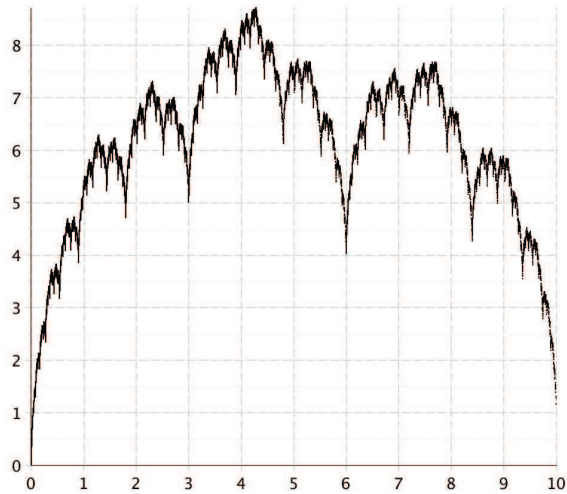}
  \caption{The attactor of the IFS $\{\mathbb{R}^2; \omega_1,\omega_2,\omega_3\}$ given in Example \ref{exampleifs}.}\label{fifexample}
  \end{center}
  \end{figure}
  \end{example}

 There is a generalization of iterated function systems, the so-called graph-directed iterated function systems (GIFS). In this work, we generalize the notion of fractal interpolation function to the graph-directed case. In modelling with fractal interpolation, graph directed fractals can have some advantages since they give us a chance to code several data sets in an interpolation function. An interpolant obtained as graph directed fractal interpolation function contains some information from each of the data sets under suitable conditions.

We briefly summarize the graph-directed iterated function systems.

 Let $\mathcal{G}=(V,E)$ be a directed graph where $V$ is the set of vertices and $E$ is the set of edges. For $\alpha, \beta \in V$, the set of edges from $\alpha$ to $\beta$ is denoted by $E^{\alpha \beta}$ and its elements by $e_i^{\alpha \beta}$,  $i=1,2,...,K^{\alpha\beta}$ where $K^{\alpha\beta}$ denotes the number of elements of $E^{\alpha\beta}$. We assume that $\displaystyle E^{\alpha}:=\bigcup\limits_{\beta \in V} E^{\alpha \beta}\neq \emptyset$.

Let $\{ X^\alpha \,| \, \alpha \in V \}$ be a finite collection of complete metric spaces and $\varphi_i^{\alpha\beta}: X^\beta\to X^\alpha$ be a contraction mapping corresponding to the edge $e_i^{\alpha \beta}$ (but in the opposite direction of $e_i^{\alpha \beta}$) for $i=1,2,\dots,K^{\alpha\beta}$ and $\alpha, \beta \in V$. It can be shown that there exist a unique family of nonempty compact sets $A^\alpha\subset X^\alpha$ such that
\[
A^\alpha=\bigcup_{\beta\in V} \bigcup_{i=1}^{K^{\alpha\beta}} \varphi_i^{\alpha\beta}(A^{\beta})
\]
(see \cite{Edgar}). The system $\left\{  X^\alpha; \varphi_i^{\alpha\beta}\right\}$ is called a graph-directed iterated function system (GIFS) realizing the graph  $\mathcal{G}$ and $  A^\alpha$ 's $(\alpha \in V)$ are called the attractors of the system (see \cite{Edgar},\cite{MaWil}  for more details).

 \begin{example}
 Let $\mathcal{G}=(V,E)$ be a directed graph with $V=\{1\emph{},2\}$ as shown in Figure~\ref{exgraph}(a).
 Let $X^1$ be the unite square and $X^2$ be the right triangle as indicated in Figure \ref{exgraph}(b). Consider the contractions $\omega_i^{\alpha \beta}: X^\beta
\rightarrow X^\alpha$ ($\alpha,\beta  \in V, \ i=1,...,K^{\alpha\beta}$) corresponding to each edge of the graph (in the opposite direction). A pictorial description of the maps are shown in  Figure \ref{exgraph}(c), the explicit expressions of them can be found in \cite{Demir}.  The attractors of this graph-directed system realizing the graph  $\mathcal{G}$ look like as given in Figure \ref{exgraph}(d).
\begin{figure}[h]
\begin{center}
\centering \includegraphics[scale=0.9]{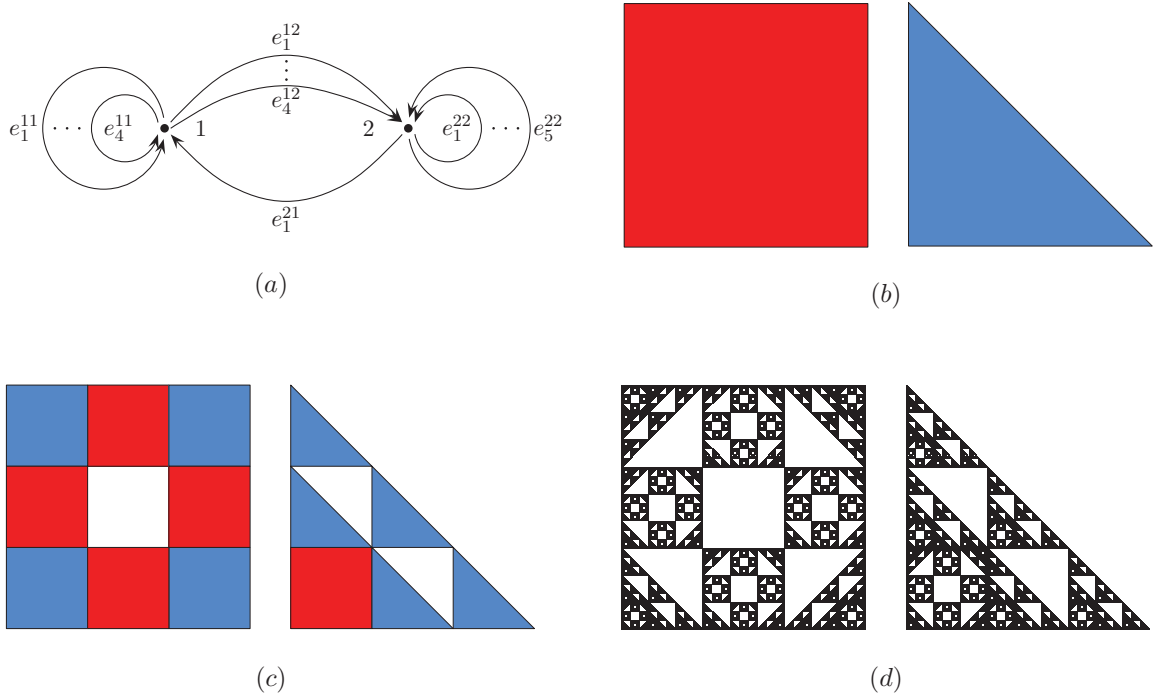}
\caption{(a) A directed graph $\mathcal{G}=(V,E)$ where $V=\{1,2\}$.
(b) The spaces $X^1$ and $X^2$. (c) Pictorial description of the  contractions $\omega_i^{\alpha,\beta}$ which is determined by a copy of $X^\beta$ in $X^\alpha$ for all $ \alpha,\beta  \in V, \ i=1,...,K^{\alpha\beta}$.
(d) Sixth stage of the iteration of the GIFS.} \label{exgraph}
\end{center}
\end{figure}
\end{example}

\section{Construction of Graph-Directed Fractal Interpolation Functions}\label{construction}
In this section, we will construct a graph-directed iterated function system for a finite number of data sets such that each attractor of the system is the graph of an interpolation function (which interpolates a corresponding data set) what we call a graph-directed fractal interpolation function.

Let
\begin{equation}\label{datasets}
\mathfrak{D}^{\alpha}=\{(x_0^\alpha,F_0^\alpha), (x_1^\alpha,F_1^\alpha), \dots, (x_{N_\alpha}^\alpha,F_{N_\alpha}^\alpha)\} \end{equation}
be data sets in $\mathbb{R}^2$ with $N_{\alpha}\geq 2$ for all $\alpha=1,2,\dots,n$.
We assume that these data sets satisfy
\begin{equation}\label{conditionondatas}
\dfrac{x_{j}^\beta-x_{j-1}^\beta}{x_{N_\alpha}^\alpha-x_0^\alpha}<1
\end{equation}
for all $\alpha\neq \beta, \, \alpha,\beta=1,\dots,n$ and $j=1,\dots,N_\beta$.

\begin{remark}
The condition (\ref{conditionondatas}) on the data sets will provide that the maps to be constructed further in this section are contractions. Note that if the original data sets do not satisfy the condition (\ref{conditionondatas}), one can make the data sets to satisfy the condition by adding new points artificially. Instead of changing the original data sets by adding artificial points, we prefer the data sets to satisfy the condition (\ref{conditionondatas}).
\end{remark}

Now we can state the following theorem as the main result of this paper.
\begin{theorem}\label{maintheorem}
Let $\mathfrak{D}^{\alpha}=\{(x_0^\alpha,F_0^\alpha),(x_1^\alpha,F_1^\alpha),\dots,(x_{N_\alpha}^\alpha,F_{N_\alpha}^\alpha)\}$ be data sets in $\mathbb{R}^2$ for $\alpha=1,\dots,n$ satisfying (\ref{conditionondatas}).  Then there exists a GIFS, with attractors $A^\alpha$ ($\alpha=1,2,\dots,n$), such that $A^\alpha$ is the graph of a function which interpolates the data set $\mathfrak{D}^{\alpha}$ for each $\alpha$.
\end{theorem}

We call an interpolation function whose graph is the attractor of a graph-directed iterated function system a graph-directed fractal interpolation function.

In the rest of this section we deal with the construction of a GIFS satisfying Theorem~\ref{maintheorem} and for simplicity we consider the case $n=2$. In Section \ref{main}, we prove that Theorem \ref{maintheorem} holds for a GIFS to be constructed in this section.

Let
\begin{eqnarray*}
\mathfrak{D}^1&=&\{ (x_0,F_0), (x_1,F_1), \dots, (x_N,F_N)\} \\
\mathfrak{D}^2&=&\{ (y_0,G_0),(y_1,G_1),\dots,(y_M,G_M)\}
\end{eqnarray*}
be two data sets in $\mathbb{R}^2$ with $N,M \geq 2$.
In the case of $n=2$ the condition (\ref{conditionondatas}) reduces
\begin{equation}\label{conditionondataiki}
\dfrac{x_{i}-x_{i-1}}{y_{M}-y_0}<1 \ \text{ and } \ \dfrac{y_{j}-y_{j-1}}{x_{N}-x_0}<1
\end{equation}
for all $i=1,2,\dots,N$ and $j=1,2,\dots,M$.

\noindent Clearly, one can form a directed graph  $\mathcal{G}=(V,E)$ with $V=\{1,2\}$ such that
\begin{equation}\label{ikilikosul}
K^{11}+K^{12}=N \mbox{ and } K^{21}+K^{22}=M
\end{equation}
\noindent(recall that $K^{\alpha\beta}$ is the number of elements of $E^{\alpha\beta}$ for $\alpha,\beta \in V$).
Such a graph is shown in Figure \ref{graph2ligenel}.

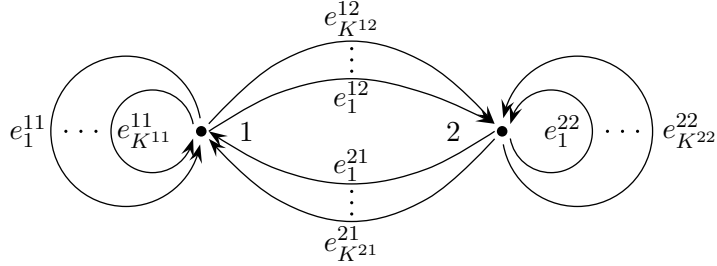
\begin{figure}[h]
\begin{center}
\begin{pspicture}(-3,-2)(7,2)
\pscurve[linewidth=0.5pt,arrowsize=5pt]{->}(0.1,0)(1,0.5)(2,0.7)(3,0.5)(3.9,0.1)
\pscurve[linewidth=0.5pt,arrowsize=5pt]{->}(0.1,0.1)(1,0.9)(2,1.2)(3,0.9)(3.9,0.1)
\psarcn[linewidth=0.5pt,arrowsize=5pt]{<-}(4.65,0){0.55}{170}{190}
\psarcn[linewidth=0.5pt,arrowsize=5pt]{<-}(5,0){1}{170}{190}
\psarcn[linewidth=0.5pt,arrowsize=5pt]{<-}(-0.65,0){0.55}{350}{10}
\psarcn[linewidth=0.5pt,arrowsize=5pt]{<-}(-1,0){1}{350}{10}
\pscurve[linewidth=0.5pt,arrowsize=5pt]{<-}(0.1,0)(1,-0.5)(2,-0.7)(3,-0.5)(3.9,0)
\pscurve[linewidth=0.5pt,arrowsize=5pt]{<-}(0.1,-0.1)(1,-0.9)(2,-1.2)(3,-0.9)(3.9,-0.1)
\qdisk(0,0){2pt} \qdisk(4,0){2pt}
\rput(-2.3,0){$e_1^{11}$}\rput(-0.75,0){$e_{K^{11}}^{11}$}
\rput(4.8,0){$e_{1}^{22}$} \rput(6.5,0){$e_{K^{22}}^{22}$}
\rput(2,1.5){$e_{K^{12}}^{12}$} \rput(2,0.45){$e_1^{12}$}
\rput(2,-1.5){$e_{K^{21}}^{21}$} \rput(2,-0.45){$e_1^{21}$}
\rput(0.6,0){$1$}
\rput(3.35,0){$2$}
\qdisk(5.8,0){0.6pt}\qdisk(5.6,0){0.6pt}\qdisk(5.4,0){0.6pt}
\qdisk(-1.8,0){0.6pt}\qdisk(-1.6,0){0.6pt}\qdisk(-1.4,0){0.6pt}
\qdisk(2,0.8){0.6pt}\qdisk(2,0.95){0.6pt}\qdisk(2,1.1){0.6pt}
\qdisk(2,-0.8){0.6pt}\qdisk(2,-0.95){0.6pt}\qdisk(2,-1.1){0.6pt}
\end{pspicture}
\caption{A directed graph $(V,E)$ where $V=\{1,2\}$.}\label{graph2ligenel}
\end{center}
\end{figure}

\begin{remark}
For the general case, the condition (\ref{ikilikosul}) on the number of edges between the vertices of the graph will be \[ K^{\alpha 1}+K^{\alpha 2}+\cdots+K^{\alpha n}=N_{\alpha}\] for $\alpha=1,2,\dots,n$.
\end{remark}
\begin{remark}
Note that we allow the case $K^{\alpha \beta}=0$. This may be lead to a situation that a data set can not be coded in an interpolant. If one construct a strongly connected graph (i.e. there is a directed path from the initial vertex $\alpha$ to the terminal vertex $\beta$ for each pair $\alpha, \beta \in V$ of the graph) satisfying the condition (\ref{ikilikosul}) then each interpolant will nestle some information from each of the data sets and thus the graph-directed self-similarity of the interpolants become more genuine.
\end{remark}

To obtain a GIFS associated with the data sets $\mathfrak{D}^{\alpha}$, $(\alpha=1,2)$ and realizing the graph $\mathcal{G}$ consider the affine functions $\omega_i^{\alpha\beta}: \mathbb{R}^2 \to \mathbb{R}^2$, $i=1,2,\dots,K^{\alpha\beta}$  of the form
\[\omega_i^{\alpha\beta}(x,y)= \left(
   \begin{array}{cc}
     a_i^{\alpha\beta} & 0 \\
     c_i^{\alpha\beta} & d_i^{\alpha\beta}\\
   \end{array}
 \right)
  \left(
    \begin{array}{c}
      x\\
      y \\
    \end{array}
  \right)+\left(
    \begin{array}{c}
      e_i^{\alpha\beta}\\
      f_i^{\alpha\beta} \\
    \end{array}
  \right)
  \] where the $a_i^{\alpha\beta},c_i^{\alpha\beta},d_i^{\alpha\beta},e_i^{\alpha\beta},f_i^{\alpha\beta}$ are all real numbers for $\alpha, \beta \in \{1,2\}$.

  \noindent We constrain these functions as follows:

\bigskip

    \noindent $\bullet$
   $ \left\{
    \begin{array}{l}
    w_i^{11}(x_0,F_0)=(x_{i-1},F_{i-1}) \\
    w_i^{11}(x_N,F_N)=(x_{i},F_{i})
    \end{array} \right. $
     for $i=1,2,\dots,K^{11}$

\bigskip

    \noindent $\bullet$
    $ \left\{
    \begin{array}{l}
    w_{i-K^{11}}^{12}(y_0,G_0)=(x_{i-1},F_{i-1})  \\ w_{i-K^{11}}^{12}(y_M,G_M)=(x_{i},F_{i})
    \end{array} \right. $
     for $i=K^{11}+1,\dots,K^{11}+K^{12}=N$

\bigskip

    \noindent $\bullet$
    $ \left\{
    \begin{array}{l}
    w_i^{21}(x_0,F_0)=(y_{i-1},G_{i-1}) \\
    w_i^{21}(x_N,F_N)=(y_{i},G_{i})
    \end{array} \right. $
     for $i=1,2,\dots,K^{21}$

\bigskip

    \noindent $\bullet$
    $ \left\{
    \begin{array}{l}
    w_{i-K^{21}}^{22}(y_0,G_0)=(y_{i-1},G_{i-1}) \\
    w_{i-K^{21}}^{22}(y_M,G_M)=(y_{i},G_{i})
     \end{array} \right. $
     for $i=K^{21}+1,\dots,K^{21}+K^{22}=M$
     \bigskip

\noindent Note that each constraint should be applied in the case of $K^{\alpha \beta} \geq 1$.
\begin{remark}
Complying with the mapping relationship determined by the graph $\mathcal{G}$, one can rearrange the maps in an arbitrary order. Since our aim is to show the existence of a graph-directed interpolation function, we arrange the mappings according to the order of data sets.
\end{remark}

From each group of above conditions we get the following equation systems respectively:
  \begin{equation}\label{eq11}
   \mbox{ For all } i=1,2,\dots,K^{11} \ \ \left\{ \begin{array}{ccl}
   x_{i-1}&=&a_i^{11}x_0+e_i^{11}\\
   F_{i-1}&=&c_i^{11}x_0+d_i^{11}F_0+f_i^{11}\\
   x_{i}&=&a_i^{11}x_N+e_i^{11}\\
   F_{i}&=&c_i^{11}x_N+d_i^{11}F_N+f_i^{11}
  \end{array} \right.
\end{equation}

  \begin{equation}\label{eq12}
\mbox{ For all } i=K^{11}+1,\dots,N \ \ \left\{   \begin{array}{ccl}
   x_{i-1}&=&a_{i-K^{11}}^{12}y_0+e_{i-K^{11}}^{12}\\
   F_{i-1}&=&c_{i-K^{11}}^{12}y_0+d_{i-K^{11}}^{12}G_0+f_{i-K^{11}}^{12}\\
   x_{i}&=&a_{i-K^{11}}^{12}y_M+e_{i-K^{11}}^{12}\\
   F_{i}&=&c_{i-K^{11}}^{12}y_M+d_{i-K^{11}}^{12}G_M+f_{i-K^{11}}^{12}
  \end{array} \right.
\end{equation}

  \begin{equation}\label{eq21}
\mbox{ For all } i=1,2,\dots,K^{21} \ \ \left\{   \begin{array}{ccl}
   y_{i-1}&=&a_{i}^{21}x_0+e_{i}^{21}\\
   G_{i-1}&=&c_i^{21}x_0+d_i^{21}F_0+f_i^{21}\\
   y_{i}&=&a_i^{21}x_N+e_i^{21}\\
   G_{i}&=&c_i^{21}x_N+d_i^{21}F_N+f_i^{21}
  \end{array} \right.
\end{equation}

  \begin{equation}\label{eq22}
 \mbox{ For all } i=K^{21}+1,\dots,M \ \ \left\{  \begin{array}{ccl}
   y_{i-1}&=&a_{i-K^{21}}^{22}y_0+e_{i-K^{21}}^{22}\\
   G_{i-1}&=&c_{i-K^{21}}^{22}y_0+d_{i-K^{21}}^{22}G_0+f_{i-K^{21}}^{22}\\
   y_{i}&=&a_{i-K^{21}}^{22}y_M+e_{i-K^{21}}^{22}\\
   G_{i}&=&c_{i-K^{21}}^{22}y_M+d_{i-K^{21}}^{22}G_M+f_{i-K^{21}}^{22}
  \end{array} \right.
\end{equation}

Choosing $d_i^{\alpha\beta}$ as a parameter, one can easily solve the linear equation systems (\ref{eq11})--(\ref{eq22}) to obtain $a_i^{\alpha\beta},c_i^{\alpha\beta},e_i^{\alpha\beta},f_i^{\alpha\beta}$ for $\alpha,\beta \in \{1,2\}, \, i=1,2,\dots,K^{\alpha \beta}$  in terms of the data points and $d_i^{\alpha\beta}$ as follows:
\renewcommand{\arraystretch}{1.5}
\[
\begin{tabular}{cc}
$
\begin{array}{l}
a_i^{11}=\dfrac{x_i-x_{i-1}}{x_N-x_0}\\
e_i^{11}=\dfrac{x_N x_{i-1}-x_0x_{i}}{x_N-x_0}\\
c_i^{11}=\dfrac{F_i-F_{i-1}}{x_N-x_0}-d_i^{11} \dfrac{F_N-F_{0}}{x_N-x_0}\\
f_i^{11}=\dfrac{x_N F_{i-1}-x_0F_{i}}{x_N-x_0}-d_i^{11} \dfrac{x_N F_0-x_0F_N}{x_N-x_0}
\end{array}
$
&
$
\begin{array}{l}
a_i^{12}=\dfrac{x_i-x_{i-1}}{y_M-y_0}\\
e_i^{12}=\dfrac{y_M x_{i-1}-y_0x_{i}}{y_M-y_0}\\
c_i^{12}=\dfrac{F_i-F_{i-1}}{y_M-y_0}-d_i^{12} \dfrac{G_M-G_0}{y_M-y_0}\\
f_i^{12}=\dfrac{y_M F_{i-1}-y_0F_{i}}{y_M-y_0}-d_i^{12} \dfrac{y_M G_0-y_0G_M}{y_M-y_0}
\end{array}
$
\\ \\
$
\begin{array}{l}
a_i^{21}=\dfrac{y_i-y_{i-1}}{x_N-x_0}\\
e_i^{21}=\dfrac{x_N y_{i-1}-x_0y_{i}}{x_N-x_0}\\
c_i^{21}=\dfrac{G_i-G_{i-1}}{x_N-x_0}-d_i^{21} \dfrac{F_N-F_{0}}{x_N-x_0}\\
f_i^{21}=\dfrac{x_N G_{i-1}-x_0G_{i}}{x_N-x_0}-d_i^{21} \dfrac{x_N F_0-x_0F_N}{x_N-x_0}
\end{array}
$
&
$
\begin{array}{l}
a_i^{22}=\dfrac{y_i-y_{i-1}}{y_M-y_0}\\
e_i^{22}=\dfrac{y_M y_{i-1}-y_0 y_{i}}{y_M-y_0}\\
c_i^{22}=\dfrac{G_i-G_{i-1}}{y_M-y_0}-d_i^{22} \dfrac{G_M-G_0}{y_M-y_0}\\
f_i^{22}=\dfrac{y_M G_{i-1}-y_0 G_{i}}{y_M-y_0}-d_i^{22} \dfrac{y_M G_0-y_0G_M}{y_M-y_0}
\end{array}
$
\end{tabular}
\]
 \bigskip

To show that the obtained system $\left\{ \mathbb{R}^2 ;\omega_i^{\alpha\beta} \right\}$ is a GIFS, we need to ensure that $\omega_i^{\alpha\beta}$ is a contraction for each $i$, $\alpha$ and $\beta$. Although $w_i^{\alpha\beta}$ 's need not be contractions with respect to the Euclidean metric on $\mathbb{R}^2$, the following lemma guarantees that they are contractions with respect to a metric which is equivalent to the Euclidean one.

\begin{lemma}\label{lemma1}
Let $w_i^{\alpha\beta}$ be the affine functions constructed above associated with the data sets $\mathfrak{D}^{\alpha}$, $(\alpha=1,2)$ which satisfy (\ref{conditionondataiki}). Let $| d_i^{\alpha\beta}|<1$ for all $\alpha,\beta \in \{1,2\}$ and $i=1,2,\dots,K^{\alpha \beta}$. Then the system $\left\{\mathbb{R}^2, w_i^{\alpha\beta}\right\}$ is a graph-directed iterated function system associated with the data sets $\mathfrak{D}^{\alpha}$, $(\alpha=1,2)$ and realizing the graph $\mathcal{G}$.
\end{lemma}

\begin{proof}  Condition (\ref{conditionondataiki}) and  \cite[Theorem 2.1]{BarnsleyBook} make each of $\omega_i^{\alpha\beta}$ a contraction with respect to a metric which is equivalent to the Euclidean metric. See \cite[Theorem 2.1]{BarnsleyBook} for more details.
\end{proof}

Depending on the parameters $d_i^{\alpha\beta}$, we have infinitely many GIFS $\left\{ \mathbb{R}^2 ;\omega_i^{\alpha\beta}\right\}$ realizing the graph $\mathcal{G}$.


\begin{example}\label{gifsexample}
   Consider the GIFS associated with data sets
   \begin{eqnarray*}
   D^1&=&\{(0,5), (1,4), (2,1), (3,1), (4,4), (5,5)\}\\
   D^2&=&\{(0,1), (1,2), (2,3), (3,2), (4,1)\}
   \end{eqnarray*}
   realizing the graph  with $K^{11}=3$, $K^{12}=2$, $K^{21}=1$, $K^{22}=3$ as shown in Figure~\ref{sonexamplegraph}.
   We choose $d_i^{\alpha\beta}=\dfrac{1}{3}$ for all $\alpha,\beta \in \{1,2\}$ and $i=1,...,K^{\alpha\beta}$. The attractors $A^1$ and $A^2$ of the system are given in Figure~\ref{gfifexample}. $A^1$ (resp $A^2$) is the the graph of a function which interpolates $D^1$ (resp. $D^2$).

  We also present the attractors of another GIFS which is obtained by changing only the vertical scaling factors as $d_1^{11}=d_2^{11}=d_3^{11}=0.25, d_1^{12}=d_2^{12}=\frac{1}{3}, d_1^{21}=\frac{1}{4}, d_1^{22}=d_2^{22}=d_3^{22}=\frac{1}{2}$ (see Figure \ref{gfifexampleek}).

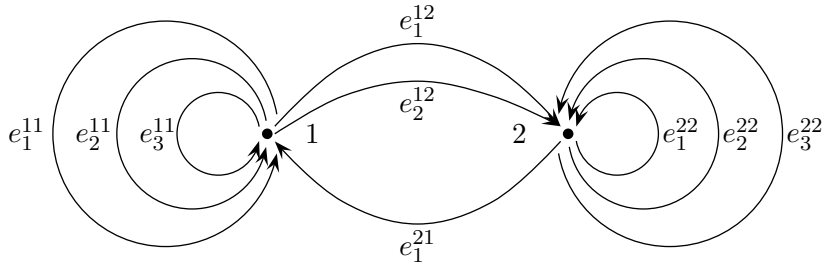
\begin{figure}[h]
\begin{center}
\begin{pspicture}(-3,-2)(7,2)
\pscurve[linewidth=0.5pt,arrowsize=5pt]{->}(0.1,0)(1,0.5)(2,0.7)(3,0.5)(3.9,0.1)
\pscurve[linewidth=0.5pt,arrowsize=5pt]{->}(0.1,0.1)(1,0.9)(2,1.2)(3,0.9)(3.9,0.1)
\psarcn[linewidth=0.5pt,arrowsize=5pt]{<-}(4.65,0){0.55}{170}{190}
\psarcn[linewidth=0.5pt,arrowsize=5pt]{<-}(5,0){1}{170}{190}
\psarcn[linewidth=0.5pt,arrowsize=5pt]{<-}(5.35,0){1.5}{170}{190}

\psarcn[linewidth=0.5pt,arrowsize=5pt]{<-}(-0.65,0){0.55}{350}{10}
\psarcn[linewidth=0.5pt,arrowsize=5pt]{<-}(-1,0){1}{350}{10}
\psarcn[linewidth=0.5pt,arrowsize=5pt]{<-}(-1.35,0){1.5}{350}{10}
\pscurve[linewidth=0.5pt,arrowsize=5pt]{<-}(0.1,-0.1)(1,-0.9)(2,-1.2)(3,-0.9)(3.9,-0.1)
\qdisk(0,0){2pt} \qdisk(4,0){2pt}
\rput(-3.2,0){$e_1^{11}$}\rput(-2.3,0){$e_2^{11}$}\rput(-1.45,0){$e_3^{11}$}
\rput(5.5,0){$e_{1}^{22}$} \rput(6.3,0){$e_{2}^{22}$}\rput(7.15,0){$e_{3}^{22}$}
\rput(2,1.5){$e_{1}^{12}$} \rput(2,0.4){$e_2^{12}$}
\rput(2,-1.5){$e_{1}^{21}$}
\rput(0.6,0){$1$}
\rput(3.35,0){$2$}
\end{pspicture}
\caption{The directed graph used in Example \ref{gifsexample}.}\label{sonexamplegraph}
\end{center}
\end{figure}

\begin{figure}[ht]
\begin{minipage}[b]{0.47\linewidth}
\centering
\includegraphics[width=7cm]{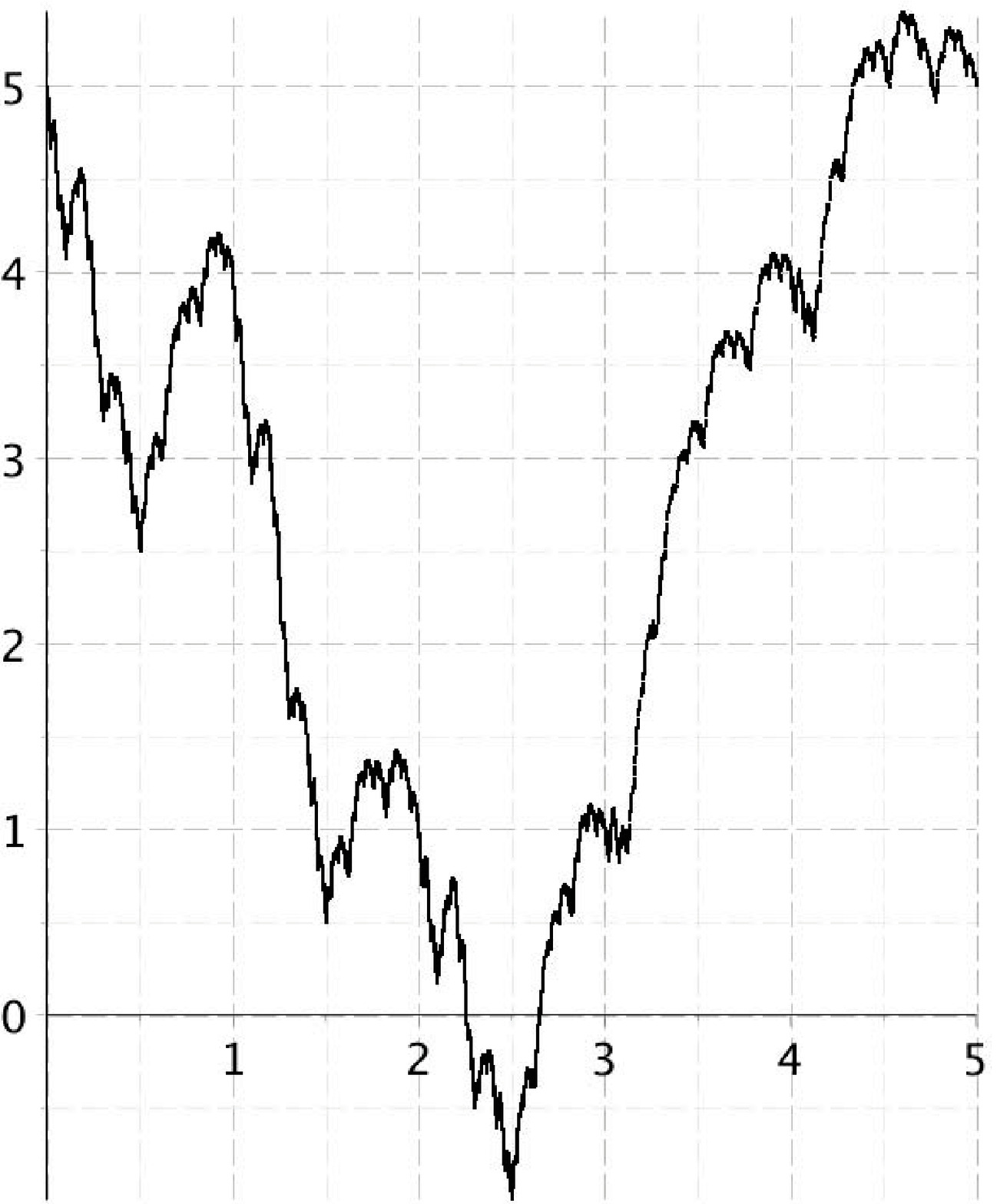}
\end{minipage}
\hspace{1cm}
\begin{minipage}[b]{0.47\linewidth}
\centering
\includegraphics[width=7cm]{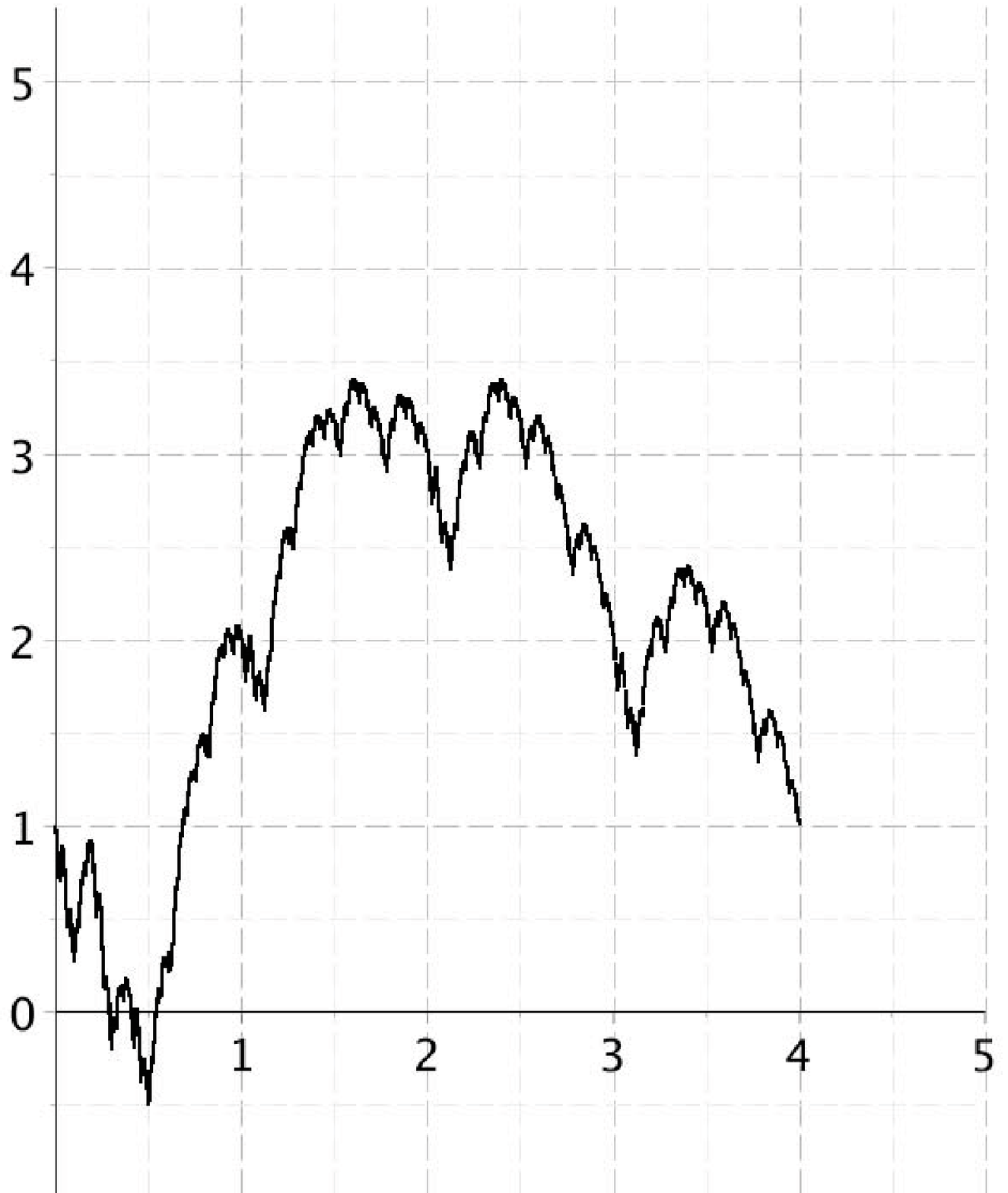}
\end{minipage}
\caption{The attractors $A^1$ (left) and $A^2$ (right) of the first GIFS given in Example \ref{gifsexample}.}\label{gfifexample}
\end{figure}

\begin{figure}[ht]
\begin{minipage}[b]{0.47\linewidth}
\centering
\includegraphics[width=7cm]{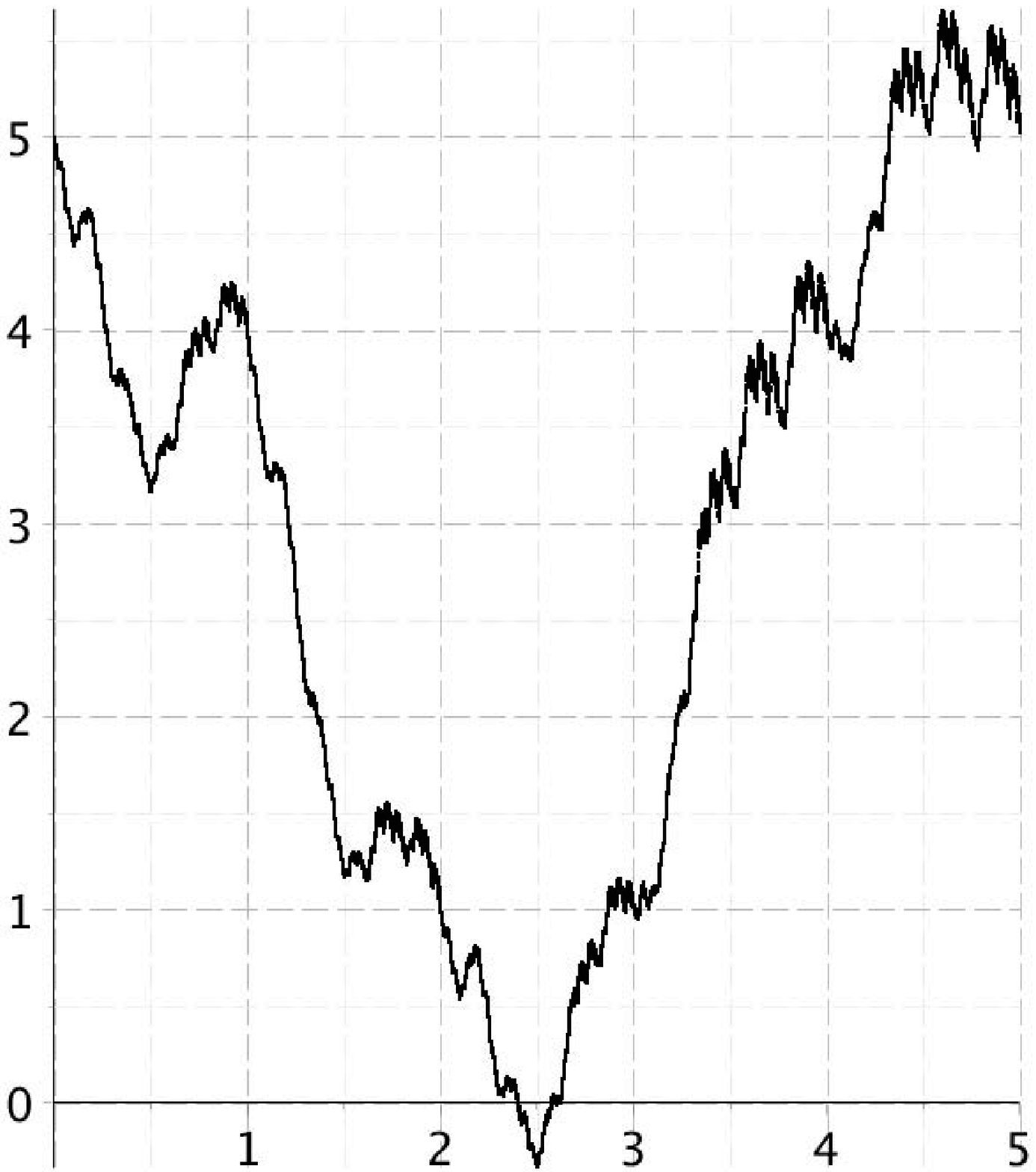}
\end{minipage}
\hspace{1cm}
\begin{minipage}[b]{0.47\linewidth}
\centering
\includegraphics[width=7cm]{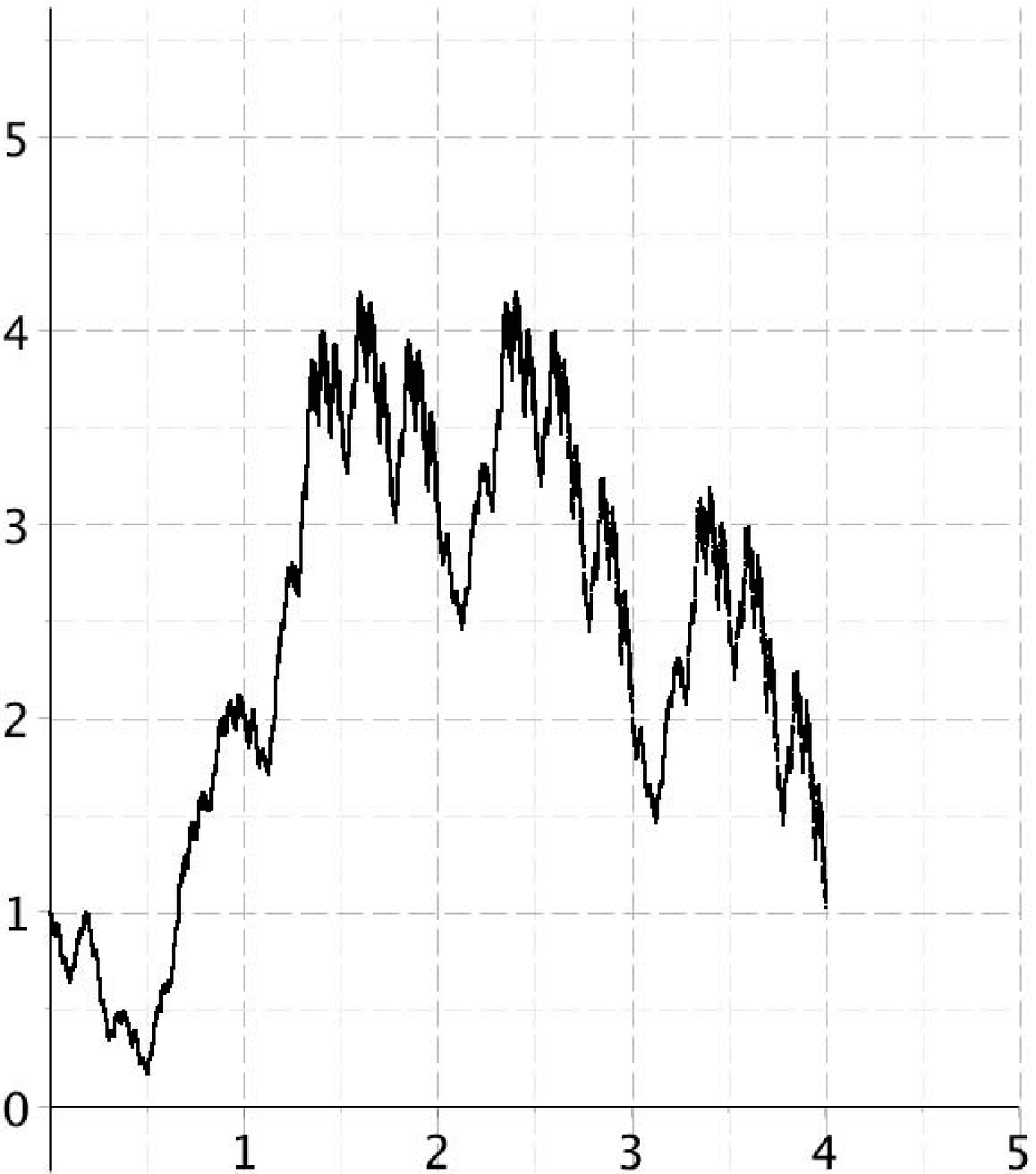}
\end{minipage}
\caption{The attractors of the second GIFS given in Example \ref{gifsexample} with different vertical scaling factors.}\label{gfifexampleek}
\end{figure}
\end{example}

\clearpage

\section{Proof of the Main Theorem}\label{main}
In this section, we prove that a GIFS constructed in Section~\ref{construction} satisfies Theorem \ref{maintheorem}. As mentioned before, we consider the case $n=2$.

Specifying the coefficients $d_i^{\alpha\beta}$ such that $|d_i^{\alpha\beta}|<1$ for all $\alpha, \beta \in V=\{1,2\}$ and $i=1,2,\dots,K^{\alpha \beta}$, let $\{ \mathbb{R}^2;\omega_i^{\alpha \beta}\}$ be a graph-directed iterated function system associated with the data sets $\mathfrak{D}^{\alpha}$, $(\alpha=1,2)$ constructed in Section \ref{construction}.  We prove that its attractor system $\{A^1,A^2\}$ is a pair of compact subsets of $\mathbb{R}^2$ such that  $A^1$ (resp. $A^2$) is the graph of a function which interpolates the data set $\mathfrak{D}^{1}$, (resp. $\mathfrak{D}^{2}$).

Let \[\mathcal{F}= \{ f \, | \, f: [x_0,x_N] \rightarrow \mathbb{R} \text{ continuous, } f(x_0)=F_0, f(x_N)=F_N\}\]
\[\mathcal{G}= \{ g \, | \, g: [y_0,y_M] \rightarrow \mathbb{R} \text{ continuous, } g(y_0)=G_0, g(y_M)=G_M\}\]
be metric spaces with the metrics
\[
d_\mathcal{F}(f_1,f_2)=\max\{\, |f_1(x)-f_2(x)| \ | \, x\in [x_0,x_N]\, \} \]
and
\[d_\mathcal{G}(g_1,g_2)=\max\{\, |g_1(x)-g_2(x)| \ | \, x\in [y_0,y_M]\, \} \] respectively. Note that $(\mathcal{F}, d_\mathcal{F})$ and $(\mathcal{G},d_\mathcal{G})$ are complete metric spaces, hence $\mathcal{F}\times \mathcal{G}$ is also complete with the metric $d$ where
\[d( (f_1,g_1),(f_2,g_2))=\max \{ d_\mathcal{F}(f_1,f_2), d_\mathcal{G}(g_1,g_2)\}.\]

We will define a contraction mapping on $\mathcal{F}\times \mathcal{G}$ using the coefficients $a_i^{\alpha\beta}, c_i^{\alpha\beta},e_i^{\alpha\beta},f_i^{\alpha\beta}$ and  the vertical scaling factors $d_i^{\alpha\beta} <1$ of the GIFS $\{ \mathbb{R}^2;\omega_i^{\alpha \beta}\}$  for $\alpha,\beta \in \{1,2\}$ and $i=1,2,\dots,K^{\alpha \beta}$.
Employing the following affine invertible functions
\begin{eqnarray*}
I_i: [x_0,x_N] \to [x_{i-1},x_i], \, I_i(x)=a_i^{11}x+e_i^{11}   & \text{ for } & i=1,...,K^{11}, \\
I_i: [y_0,y_M] \to [x_{i-1},x_i], \, I_i(x)=a_{i-K^{11}}^{12}\,x+e_{i-K^{11}}^{12}   & \text{ for } & i=K^{11}+1,...,N,\\
J_i: [x_0,x_N] \to [y_{i-1},y_i], \, J_i(x)=a_i^{21}x+e_i^{21}   & \text{ for } & i=1,...,K^{21},\\
J_i: [y_0,y_M] \to [y_{i-1},y_i], \, J_i(x)=a_{i-K^{21}}^{22}\,x+e_{i-K^{21}}^{22}   & \text{ for } & i=K^{12}+1,...,M
\end{eqnarray*}
we define the mapping
\begin{eqnarray*}
T: \mathcal{F} \times \mathcal{G} &\to& \mathcal{F} \times \mathcal{G} \\
T(f,g)(x,y)&=&(\tilde{f}(x),\tilde{g}(y))
\end{eqnarray*}
where
\[
\tilde{f}(x)=\left\{
\begin{array}{ll}
c_i^{11}I_i^{-1}(x)+d_i^{11}f(I_i^{-1}(x))+f_i^{11}  & ; \, x \in [x_{i-1},x_i] \text{ for }  i=1,...,K^{11} \\
c_{i-K^{11}}^{12}I_i^{-1}(x)+d_{i-K^{11}}^{12}g(I_i^{-1}(x))+f_{i-K^{11}}^{12}  &; \, x \in [x_{i-1},x_i] \text{ for }  i=K^{11}+1,...,N,
\end{array}
\right.
\]
\[
\tilde{g}(y)=\left\{
\begin{array}{ll}
c_j^{21}J_j^{-1}(y)+d_j^{21}f(J_j^{-1}(y))+f_j^{21}  &; \, y \in [y_{j-1},y_j] \text{ for } j=1,...,K^{21} \\
c_{j-K^{21}}^{22}J_j^{-1}(y)+d_{j-K^{21}}^{22}g(J_j^{-1}(y))+f_{j-K^{21}}^{22}    &; \, y \in [y_{j-1},y_j] \text{ for }  j=K^{21}+1,...,M.
\end{array}
\right.
\]

It can be easily seen that $\tilde{f}(x_0)=F_0$, $\tilde{f}(x_N)=F_N$, $\tilde{g}(y_0)=G_0$ and $\tilde{g}(y_M)=G_M$ using the equations (\ref{eq11})--(\ref{eq22}), thus to see that $T$ is well defined it is enough to show that $\tilde{f}$ and $\tilde{g}$ are continuous.

$\tilde{f}$ is continuous on the intervals $[x_{i-1},x_i]$ since $I_i^{-1}$ is continuous for $i=1,\dots,N$. So we must show that $\tilde{f}$ is also continuous at the points $\{x_1,x_2,\dots,x_{N-1}\}$ on which the function $\tilde{f}$ is defined by two different ways.

Note that for $i=1,2,\dots,K^{11}-1$
\begin{eqnarray*}
\tilde{f}(x_i)&=&c_i^{11}I_i^{-1}(x_i)+d_i^{11}f(I_i^{-1}(x_i))+f_i^{11} \label{ucf1} \\
\tilde{f}(x_i)&=&c_{i+1}^{11}I_{i+1}^{-1}(x_i)+d_{i+1}^{11}f(I_{i+1}^{-1}(x_i))+f_{i+1}^{11} \label{ucf2}
\end{eqnarray*}
are both equal $F_i$  since $I_i^{-1}(x_{i})=x_N$ and $I_{i+1}^{-1}(x_{i})=x_0$ by using (\ref{eq11}).

For $i=K^{11}+1,\dots, N-1$
\begin{eqnarray*}
\tilde{f}(x_i)&=&c_{i-K^{11}}^{12}I_i^{-1}(x_i)+d_{i-K^{11}}^{12}g(I_i^{-1}(x_i))+f_{i-K^{11}}^{12} \label{ucf3} \\
\tilde{f}(x_i)&=&c_{i+1-K^{11}}^{12}I_{i+1}^{-1}(x_i)+d_{i+1-K^{11}}^{12}g(I_{i+1}^{-1}(x_i))+f_{i+1-K^{11}}^{12} \label{ucf4}
\end{eqnarray*}
are both equal $F_i$  since $I_i^{-1}(x_{i})=y_M$ and $I_{i+1}^{-1}(x_{i})=y_0$ by using (\ref{eq12}).

Finally for $i=K^{11}$,
\begin{eqnarray*}
\tilde{f}(x_i)&=&c_i^{11}I_i^{-1}(x_i)+d_i^{11}f(I_i^{-1}(x_i))+f_i^{11} \label{ucf3} \\
\tilde{f}(x_i)&=&c_{i+1-K^{11}}^{12}I_{i+1}^{-1}(x_i)+d_{i+1-K^{11}}^{12}g(I_{i+1}^{-1}(x_i))+f_{i+1-K^{11}}^{12} \label{ucf4}
\end{eqnarray*}
are also both equal to $F_i$ since $I_i^{-1}(x_{i})=x_N$ and $I_{i+1}^{-1}(x_{i})=y_0$ by using (\ref{eq11}) and (\ref{eq12}) which shows the continuity of $\tilde {f}$.

Similarly, one can show that $\tilde{g}$ is continuous.

To see that $T$ is a contraction, we need to show that
\[d(T(f_1,g_1), T(f_2,g_2)) \leq r \, \max\{d_\mathcal{F}(f_1,f_2),d_\mathcal{G}(g_1,g_2) \}\]
for some $0<r<1$. Let $T(f_1,g_1)=(\tilde{f_1},\tilde{g_1})$ and $T(f_2,g_2)=(\tilde{f_2},\tilde{g_2})$.
Since
\begin{eqnarray*}
\max_{x\in [x_0,x_{K^{11}}]}\{|\tilde{f_1}(x)-\tilde{f_2}(x)|\}& =& \max_{i=1,...,K^{11}}\{|d_i^{11}| \ |f_1\left(I_i^{-1}(x)\right)- f_2\left(I_i^{-1}(x)\right)| \, \mid \, x\in[x_{i-1},x_i]\}\\&\leq & |d^{11}| \, d_\mathcal{F}(f_1,f_2)\\
\max_{x\in [x_{K^{11}},x_N]}\{|\tilde{f_1}(x)-\tilde{f_2}(x)|\} &=& \max_{i=K^{11}+1,...,N}\{|d_{i-K^{11}}^{12}| \ |g_1\left(I_i^{-1}(x)\right)- g_2\left(I_i^{-1}(x)\right)|\, \mid \, x\in[x_{i-1},x_i]\}\\
&\leq &|d^{12}| \, d_\mathcal{G}(g_1,g_2)
\end{eqnarray*}
where $d^{11}=\max \{ |d_i^{11}| \, \mid \, i=1,\dots,K^{11}\} $ and $d^{12}=\max \{ |d_i^{12}| \, \mid \, i=1,\dots,K^{12}\} $, we get
\[
d_\mathcal{F}(\tilde{f_1},\tilde{f_2}) \leq \max\{d^{11}, d^{12}\} \, \max\{ d_\mathcal{F}(f_1,f_2) , d_\mathcal{G}(g_1,g_2) \}.
\]

\noindent Similarly, we have
\[
d_\mathcal{G}(\tilde{g_1},\tilde{g_2}) \leq \max\{d^{21}, d^{22}\} \, \max\{ d_\mathcal{F}(f_1,f_2) , d_\mathcal{G}(g_1,g_2) \},\]
where $d^{21}=\max \{ |d_i^{21}| \, \mid \, i=1,\dots,K^{21}\} $ and $d^{22}=\max \{ |d_i^{22}| \, \mid \, i=1,\dots,K^{22}\} $. Using these results, we write
\[
d(T(f_1,g_1), T(f_2,g_2)) =  \max\{  d_\mathcal{F}(\tilde{f_1},\tilde{f_2}),  d_\mathcal{G}(\tilde{g_1},\tilde{g_2})\}\leq  r \, \max\{ d_\mathcal{F}(f_1,f_2) , d_\mathcal{G}(g_1,g_2)\},
\]
where $r=\max\{d^{11}, d^{12},d^{21}, d^{22}\}$ which is less than $1$.

From the Banach fixed point theorem,  $T$ has a unique fixed point, say $(f_0,g_0)$, that is $T(f_0,g_0)=(f_0,g_0)$. Let $F$ and $G$ be the graphs of $f_0$ and $g_0$ respectively.
Notice that for $x \in [x_0,x_N]$ and $y \in [y_0,y_M]$
\begin{eqnarray*}
f_0(a_i^{11}x+e_i^{11})= c_i^{11}x+d_i^{11}f_0(x)+f_i^{11} & \text{ for } &  i=1,...,K^{11}   \\
f_0(a_i^{12}y+e_i^{12})=c_i^{12}y+d_i^{12}g_0(y)+f_i^{12}   & \text{ for } &  i=1,...,K^{12}
\end{eqnarray*}
and
\begin{eqnarray*}
g_0(a_i^{21}x+e_i^{21})= c_i^{21}x+d_i^{21}f_0(x)+f_i^{21} & \text{ for } &  i=1,...,K^{21} \\
g_0(a_i^{22}y+e_i^{22})=c_i^{22}y+d_i^{22}g_0(y)+f_i^{22}   & \text{ for } &  i=1,...,K^{22}
\end{eqnarray*}
which imply that
\begin{eqnarray*}
F&=&\bigcup_{i=1}^{K^{11}} \omega_i^{11}(F) \cup  \bigcup_{i=1}^{K^{12}} \omega_i^{12}(G)\\
G&=&\bigcup_{i=1}^{K^{21}} \omega_i^{21}(F) \cup  \bigcup_{i=1}^{K^{22}} \omega_i^{22}(G).
\end{eqnarray*}
Since this equations system has a solution $(A^1,A^2)$, it must be $F=A^1$ and $G=A^2$ from the uniqueness of the solution. Thus the graphs of the fractal interpolation functions of $f_0$ and $g_0$ are the attractors $A^1$ and $A^2$ respectively.

\end{document}